\documentclass[10pt,oneside]{amsart}
\usepackage{verbatim}
\usepackage{amssymb}
\newcommand{\Z}{\mathbb Z}
\newcommand{\GL}{{\text {\rm GL}}}
\newcommand{\SL}{{\text {\rm SL}}}
\newcommand{\smat}[4]{\left(\begin{smallmatrix}
                 #1 & #2\\
                 #3 & #4
\end{smallmatrix}\right)}

\usepackage{amsthm}
\newtheorem{theorem}{Theorem}[section]
\newtheorem{lemma}[theorem]{Lemma}
\newtheorem{proposition}[theorem]{Proposition}
\begin{document}
\title[Coefficients of weakly holomorphic modular forms]{Two-divisibility of the coefficients of certain weakly
holomorphic modular forms}
\author{Darrin Doud}
\address{Department of Mathematics, Brigham Young University, Provo, UT  84602}
\email{doud@math.byu.edu}
\author{Paul Jenkins}
\address{Department of Mathematics, Brigham Young University, Provo, UT  84602}
\email{jenkins@math.byu.edu}
\author{John Lopez}
\address{Department of Mathematics, Brigham Young University, Provo, UT  84602}
\email{johnvlopez@gmail.com}
\date{\today}

\begin{abstract}
We study a canonical basis for spaces of weakly holomorphic modular
forms of weights 12, 16, 18, 20, 22, and 26 on the full modular
group. We prove a relation between the Fourier coefficients of
modular forms in this canonical basis and a generalized Ramanujan
$\tau$-function, and use this to prove that these Fourier
coefficients are often highly divisible by 2.
\end{abstract}

\maketitle
\section{Introduction}
The unique cusp form $\Delta(z)$ of weight 12 for $\SL_2(\Z)$ has a
Fourier expansion given by
$$\Delta(z)=q\prod_{n=1}^\infty(1-q^n)^{24} =\sum_{n=1}^\infty \tau(n)q^n,$$
where $q=e^{2\pi iz}$; the Fourier coefficients of $\Delta(z)$ are
the values of the Ramanujan $\tau$-function.  Since $\Delta(z)$ is a
Hecke eigenform, the function $\tau(n)$ is multiplicative, so that
$\tau(mn) = \tau(m)\tau(n)$ for $(m,n)=1$, and for a prime $p$ and a
positive integer $a \geq 2$, we have \[\tau(p^a) =
\tau(p)\tau(p^{a-1}) - p^{11} \tau(p^{a-2}).\]

Congruences for the values of $\tau(n)$ have been studied by many
authors, as detailed by Swinnerton-Dyer in \cite{SwD2, SwD}. For
instance, Ramanujan \cite{Ramanujan}, \cite[p. 165]{Hardy} showed
that \[\tau(2n) \equiv 0 \pmod{2},\]
\[\tau(3n) \equiv 0 \pmod{3},\]
\[\tau(5n) \equiv 0 \pmod{5},\] and that \[\tau(7n), \tau(7n+3),
\tau(7n+5), \tau(7n+6) \equiv 0 \pmod{7}.\]  More generally, he
showed \cite{BerndtOno} that with $\sigma_k(n) = \sum_{d|n} d^k$,
\[\tau(2n+1) \equiv \sigma_{11}(2n+1) \pmod{2^8},\]
\[\tau(n) \equiv n^2 \sigma_7(n) \pmod{27},\]
\[\tau(n) \equiv n \sigma_9(n) \pmod{25},\] \[\tau(n) \equiv n
\sigma_3(n) \pmod{7}.\] Swinnerton-Dyer \cite{SwD2, SwD} showed that
even stronger forms of these congruences may be derived from the
existence of certain $\ell$-adic representations $\rho_\ell :
G_\mathbb{Q} \rightarrow \GL_2(\mathbb{Z}_\ell)$; the existence of
such representations was conjectured by Serre and proved by Deligne.
The idea is that when the image of $\rho_\ell$ is ``small'', then
$\ell$-adic knowledge of the determinant of an element provides
information about the trace and thus gives a congruence for $\tau(p)
\pmod{\ell}$.  This representation can also be used in conjunction
with the Chebotarev density theorem to show that
$\tau(p)$ is usually nonzero.  
However, these methods for obtaining
congruences generally depend on the fact that $\Delta(z)$ is a cusp
form, and do not easily generalize to meromorphic modular forms.

Congruences similar to those described by Swinnerton-Dyer for
$\tau(n)$ were proved by Lehner \cite{Lehner} for the coefficients
of the weight 0 modular $j$-function \[j(z)=\frac{E_4^3(z)}{\Delta(z)}=q^{-1}+\sum_{n=0}^\infty
c(n)q^n.\] Specifically, Lehner showed that
\[c(2n) \equiv 0 \pmod{2^{11}},\]
\[c(3n) \equiv 0 \pmod{3^5},\]
\[c(5n) \equiv 0 \pmod{5^2},\]
\[c(7n) \equiv 0 \pmod{7}.\]
In fact, these divisibility properties of $c(n)$ may be strengthened
to obtain results modulo powers of these primes; Lehner proved that
for odd $n$ and for $a \geq 1$, \[c(2^an)\equiv 0\pmod{2^{3a+8}},\]
with similar congruences for the primes $3, 5$, and $7$
\cite{Lehner2, Lehner}.  Kolberg \cite{Kolberg} and Aas
\cite{Aas1,Aas2} generalized these divisibility properties to
congruences for $c(n)$ modulo even higher powers of these primes.

It is natural to ask whether such divisibility or congruence results
hold for the Fourier coefficients of other modular forms.
For large families of weakly holomorphic modular forms with small weights, this question has recently been answered affirmatively.

A weakly holomorphic modular form for a subgroup $\Gamma$ of finite
index in $\SL_2(\Z)$ is a function that is holomorphic on the upper
half plane and satisfies the modular equation $f(\gamma z) =
(cz+d)^k f(z)$ for some integer weight $k$ and for all $\gamma = \smat{a}{b}{c}{d} \in \Gamma$, with poles of finite order at the cusps of
$\Gamma$. We shall denote the space of holomorphic modular forms of
weight $k$ and level $N$ by $M_k(N)$, and the space of weakly
holomorphic modular forms of weight $k$ and level $N$ by $M_k^!(N)$.
It is well known that $M_k(N)$ is finite dimensional.

If we write $k=12\ell+k'$, with $k'\in\{0,4,6,8,10,14\}$, there is
a canonical basis $\{f_{k,m}:m\geq-\ell\}$ for $M_k^!(1)$ consisting of
forms with Fourier expansions
$$f_{k,m}(z)=q^{-m}+\sum_{n=\ell+1}^\infty a_k(m,n)q^n.$$
We define $a_k(m,n)$ to be zero when either $m$ or $n$ is nonintegral, and when $m<-\ell$ or $n<\ell+1$.
The Fourier coefficients $a_k(m,n)$ possess a remarkable duality property relating weights $k$ and $2-k$, proven by
Duke and the second author in \cite{Duke-Jenkins}. Namely, for all even $k$ and all $m,n\in\Z$,
$$a_k(m,n)=-a_{2-k}(n,m).$$
Using this duality, the first and second authors \cite{Doud-Jenkins}
proved that for $k\in\{4,6,8,10,14\}$ the coefficients $a_k(m,n)$
are often divisible by high powers of $2$, $3$, $5$, and $7$.  As an
example, for odd $m,n$ and $a,b>0$,
$$a_4(2^am,2^bn)\equiv \begin{cases}0 \pmod {2^7}&\text{if $a>b$,}\cr 0 \pmod{2^{3(b-a)+7}}&\text{if $b>a$.}\end{cases}$$
An important part of the proof is the one-dimensionality of $M_k(1)$
for these weights. Similar results were proved by Griffin in \cite{Griffin} for the weight $k=0$.

In this paper, we will prove similar divisibility results for weakly
holomorphic modular forms in weights $k$ for which $M_k(1)$ is
two-dimensional.  As is well known, these weights are
$k\in\{12,16,18,20,22,26\}$.  To be precise, for each weight $k$, we define constants $\gamma$, $\rho$, $\chi$, $\nu$, $\eta$, and $\omega$ as
follows:
$$\gamma=\begin{cases}3 &\text{ for } k=12,16 \text{ or } 20\\4& \text{ for } k=18 \text{ or }
26\\5&\text{ for } k=22, \end{cases}\qquad\rho=\begin{cases}7 &\text{ for } k=12,20 \text{ or }
22\\8& \text{ for } k=16 \text{ or } 18\\9&\text{ for } k=26, \end{cases}$$
$$\chi=\begin{cases}7 &\text{ for } k=12,20 \text{ or } 22\\9& \text{ for } k=16 \text{ or }
18\\8&\text{ for } k=26, \end{cases}\qquad\nu=\begin{cases}15 &\text{ for } k=12 \text{ or } 26\\16&
\text{ for } k=16, 18, 20 \text{ or } 22,\end{cases}$$
$$\eta=\begin{cases}12&\text{for $k=12$ or $22,$}\cr 13&\text{for $k=16$ or $20$,}\cr10&\text{for $k=18$ or $26$,}\end{cases}\qquad
\omega=\begin{cases}3&\text{for $k=18, 22$, or $26$}\cr 4&\text{for $k=12$ or $20$,}\cr6&\text{for $k=16$}\end{cases}$$
Note that although these constants depend on the weight $k$, this dependence is omitted from the
notation.  We will prove the following theorem:
\begin{theorem}\label{main} Let $m,n$ be odd positive integers, and let $a,b$ be non-negative integers. For
$k\in\{12,16,18,20,22,26\}$, we have
$$a_k(2^am,2^bn)\equiv 0\begin{cases}
\pmod{2^{\gamma b}}&\text{for $a=0$,}\qquad\qquad\quad\text{(see Lemma~\ref{i<j})}\cr
\pmod{2^{\nu}}&\text{for $a>0$, $b=0$,}\qquad\text{(see Prop.~\ref{odd-coefficients})}\cr
\pmod{2^\chi}&\text{for $a>b\geq 1$,}\qquad\quad\text{(see Prop.~\ref{lower})}\cr
\pmod{2^{\rho+\gamma(b-a)}}&\text{for $b>a\geq 1$.}\quad\qquad\text{(see Prop.~\ref{upper})}
\end{cases}$$
In addition, $a_k(-1,2^bn)\equiv 0\pmod {2^{\gamma b}}$ (see Lemma~\ref{tau}) and
$$a_k(0,2^bn)\equiv 0\begin{cases}\pmod{2^\eta}&\text{for $b=0$,}\qquad(\text{see Lemma~\ref{ak0n}}),\cr\pmod{2^\omega}&\text{for $b>0$.}\qquad(\text{see Lemma~\ref{ak02in}}),\end{cases}$$
\end{theorem}
Note that there are no results for $a=b$.  In these cases, the best possible exponent appears to be 0, since the corresponding coefficients are often odd.

The remainder of the paper is structured as follows:  In Sections 2,
3, and 4, we introduce notation for specific modular forms and
operators that will be used throughout the paper, prove a relation
between the coefficients $a_k(m, n)$ and a generalized Ramanujan
$\tau$-function, and compute the orders of poles of certain modular
forms at the cusp at $0$.  In Section 5, we review Kolberg's method
of two-dissections and use it to prove congruences for certain
coefficients. In Sections 6 and 7, we prove Theorem~\ref{main} for
the $a>b$ and $a<b$ cases, respectively.

The authors thank the referee for a careful and thoughtful review
that greatly improved the paper.  Additionally, the third author
thanks the Brigham Young University Department of Mathematics for
its generous support of his work on this project.

\section{Modular forms and Operators}
In this section we fix notation for certain operators on spaces of modular forms and for certain specific modular forms that will be used often.

\subsection{Operators on modular forms}

For $p$ a prime, and $f(z)=\sum_{n\geq n_0}a_nq^n$ a weakly
holomorphic modular  form, we define the $U_p$ operator by
$$(f|U_p)(z)=\sum_{n\geq n_0/p}a_{pn}q^n.$$
An alternative definition of $f|U_p$ that we will find useful is \cite[Thm. 4.5]{Apostol}
$$(f|U_p)(z)=\frac{1}{p}\sum_{j=0}^{p-1}f\left(\frac{z+j}p\right).$$
If $f$ has weight $k$ and level $N$, then $f|U_p$ is again a weakly
holomorphic modular  form of weight $k$ and level $N$ (if $p|N$) or
level $pN$ (if $p\nmid N$) \cite[Prop. 2.22]{Ono}.

We also define the operators $V_p$ by $(f|V_p)(z)=f(pz)$ and
the Hecke  operators $T_p$ by $(f|T_p)=(f|U_p)+p^{k-1}(f|V_p)$.
Note that $V_p$ takes modular forms of weight $k$ and level $N$ to
forms of weight $k$ and level $Np$, while $T_p$ preserves both
weight and level \cite[Prop. 2.2, Thm.  4.5]{Ono}.

\subsection{Holomorphic Modular Forms}

Let $$\sigma_k(n)=\sum_{d|n}d^k.$$  We define the Eisenstein series
$$E_4(z)=1+240\sum_{n=1}^\infty \sigma_{3}(n)q^n\quad\text{and}\quad E_6(z)=1-504\sum_{n=1}^\infty \sigma_{5}(n)q^n,$$
and we note that these are modular forms of level 1 and weights $4$ and $6$ respectively.  Next we define
$$S_4(z)=\frac{E_4(z)-E_4(2z)}{240},\qquad S_6(z)=\frac{E_6(z)-E_6(2z)}{-504}$$
$$T_4(z)=\frac{2^4E_4(2z)-E_4(z)}{2^4-1},\quad\text{and}\quad T_6(z)=\frac{2^6E_6(2z)-E_6(z)}{2^6-1}.$$
We note that each of these forms is a holomorphic modular form of level 2, and weight 4 or 6.  The forms $S_4$ and $S_6$ vanish at infinity to order 1, while the forms $T_4$ and $T_6$ vanish at 0 to order 1.  All of these forms have integer coefficients, and using the valence formula for level 2 modular forms \cite{El-Guindy}, $S_4$ and $T_4$ are easily seen to be nonvanishing in the upper half plane.

For $k\in\{12,16,18,20,22,26\}$, we have already noted that $M_k(1)$ is two-dimensional.  It is spanned by the standard Eisenstein series $E_k(z)$ and a unique normalized cusp form that is an eigenform of all the Hecke operators.  We will denote this eigenform by $\Delta_k(z)$, and its Fourier coefficients by $\tau_k(n)$.  Hence,
$$\Delta_k(z)=\sum_{n=1}^\infty \tau_k(n)q^n.$$
We note that $\Delta_{12}$ is the modular form $\Delta$ and $\tau_{12}$ is the Ramanujan $\tau$ function. Note that in addition to being the unique normalized cusp form in weight $k$, $\Delta_k$ is also the form denoted by us as $f_{k,-1}$.  Hence, we may  also write $\tau_k(n)=a_k(-1,n)$.
\subsection{Weakly holomorphic modular functions of level 2}
We note that $\Gamma_0(2)$ has cusps at 0 and at $\infty$.  Following \cite{Doud-Jenkins},  for $k<0$ we define $\alpha_k(z)$ (resp. $\theta_k$) to be the weakly holomorphic modular form of weight $k$ and level 2 that is holomorphic at $\infty$ and has a pole of minimal possible order at 0 (resp. holomorphic at 0 and has a pole of minimal possible order at $\infty$). The fact that there is a unique form with these properties is discussed in \cite{Doud-Jenkins}.  The following fact about these functions is easily proved:
\begin{theorem} Let $k<0$ be even.  If $k\equiv 2\pmod 4$, set $i=(6-k)/4$.  Then
$$\alpha_k(z)=T_6(z)/T_4(z)^i,\qquad \theta_k(z)=S_6(z)/S_4(z)^i$$
and
$$\theta_k(-1/(2z))=-2^{3-k} z^k \alpha_k(z).$$
If $k\equiv 0\pmod 4$, set $i=-k/4$.  Then
$$\alpha_k(z)=1/T_4(z)^i,\qquad \theta_k(z)=1/S_4(z)^i$$
and
$$\theta_k(-1/(2z))=2^{-k} z^k \alpha_k(z).$$
\end{theorem}
For each relevant value of $k$, we define $\mu_k$ to be the exponent
such that $\theta_k(-1/(2z))=\pm2^{\mu_k}z^k\alpha_k(z)$. We have
the following congruences on the coefficients of $\alpha_k$, which
define constants $\xi_k$.

\begin{lemma} Let $k\in\{-10,-14,-16,-18,-20,-24\}$.  Then $\alpha_k\equiv 1\pmod{ 2^{\xi_k}}$ with
$$\xi_{-10}=\xi_{-14}=\xi_{-18}=3,\quad\xi_{-20}=4,\quad\xi_{-24}=5,\quad\xi_{-16}=6.$$
\end{lemma}
\begin{proof} This follows from \cite[Lemma 6.2]{Doud-Jenkins}.\end{proof}

Finally, following Apostol \cite[pg. 87]{Apostol}, we define
$$\Phi(z)=\frac{\Delta(2z)}{\Delta(z)}$$
and $\psi(z)=1/\Phi(z)$.
We recall from \cite{Apostol} that both $\psi$ and $\Phi$ are in $M_0^!(2)$, both have integer Fourier coefficients,
$\Phi$ has a simple zero at $\infty$ and a simple pole at 0, $\psi$ has a simple pole at
$\infty$ and a simple zero at 0, and
$$\psi\left(\frac{-1}{2z}\right)=2^{12}\Phi(z).$$

\section{Hecke Operators}

Using Hecke operators, we now prove a two useful relation between certain coefficients $a_k(m,n)$ and the $\tau_k$ function.  In order to state and prove these relations, we define the symbol

$$\delta_{x,y}=\begin{cases}1&\text{if $y|x$},\cr0&\text{if $y\nmid x$}.\end{cases}$$

\begin{lemma}\label{Hecke} Let $k\in\{12,16,18,20,22,26\}$, let $m,n$ be positive integers, and let $p$ be a prime.  Then
$$a_k(m/p,n)+a_k(m,p)\tau_k(n)+p^{k-1}a_k(mp,n)=a_k(m,np)+p^{k-1}a_k(m,n/p).$$
\end{lemma}
\begin{proof} Applying the Hecke operator $T_p$ to the modular form $f_{k,m}$, we find
\begin{align*}f_{k,m}|T_p&=p^{k-1}q^{-mp}+\delta_{m,p}q^{-m/p}+a_k(m,p)q+O(q^2)\cr&=p^{k-1}f_{k,mp}+\delta_{m,p}f_{k,m/p}+a_k(m,p)\Delta_k,\end{align*}
The $n$th coefficient of this form is
$$p^{k-1}a_k(mp,n)+a_k(m/p,n)+a_k(m,p)\tau_k(n).$$
By the definition of the Hecke operator, the $n$th coefficient of $f_{k,m}|T_p$ is
$$a_k(m,np)+p^{k-1}a_k(m,n/p).$$
Hence, the lemma is proved.\end{proof}

For the second relation, we apply the $T_p$ operator twice.

\begin{lemma}\label{Hecke2} Let $k\in\{12,16,18,20,22,26\}$, let $m,n$ be positive integers, and let $p$ be a prime.  Then
\begin{align*}
p^{2k-2}a_k(mp^2,n)+&(1+\delta_{m,p})p^{k-1}a_k(m,n)+a_k(m/p^2,n)+a_k(m,p^2)\tau_k(n)\cr&=p^{2k-2}a_k(m,n/p^2)+(1+\delta_{n,p})p^{k-1}a_k(m,n)+a_k(m,np^2).\end{align*}
\end{lemma}
\begin{proof}
Applying $T_p$ to $f_{k,m}|T_p$, we obtain
\begin{align*}f_{k,m}|T_p|T_p&=p^{2k-2}q^{-mp^2}+(1+\delta_{m,p})p^{k-1}q^{-m}+\delta_{m,p^2}q^{-m/p^2}+a_k(m,p^2)q+O(q^2)\cr&=p^{2k-2}f_{k,mp^2}+(1+\delta_{m,p})p^{k-1}f_{k,m}+\delta_{m,p^2}f_{k,m/p^2}+a_k(m,p^2)\Delta_k.\end{align*}
Computing the coefficient of $q^n$ in this expression gives the left-hand side of the desired equation; computing it directly from the definition of $f_{k,m}|T_p|T_p$ gives the right hand side.

\end{proof}

For the remainder of this paper, we will use these lemmas only in the case $p=2$.

We now prove a divisibility result for the coefficients $\tau_k$.  In addition to being important in the remainder of the paper, under the identification $\tau_k(n)=a_k(-1,n)$ this result is part of Theorem~\ref{main}.
\begin{lemma}\label{tau} Let $k\in\{12,16,18,20,22,26\}$.  For a nonnegative integer $b$ and a positive odd integer $n$, we have
$$\tau_k(2^bn)\equiv 0\pmod {2^{\gamma b}}.$$
\end{lemma}
\begin{proof} We note that $\Delta_k$ is the unique normalized cusp form of weight $k$, so that it is an eigenform of all the Hecke operators.  Hence, we have (\cite[Theorem 3]{Atkin-Lehner}) that
$$\tau_k(2^{b+1}n)=\tau_k(2)\tau_k(2^bn)-2^{k-1}\tau_k(2^{b-1}n).$$
The result then follows from the fact that $\tau_k(2)$ is divisible by $2^\gamma$ by induction on $b$, using that $k-1>2\gamma$.
\end{proof}

Another way to write the result of this lemma is to note that $a_k(-1,2^bn)\equiv 0\pmod {2^{\gamma b}}$ for $b\geq0$.
\section{Operators acting on weakly holomorphic modular forms}

The standard $U_p$ and $V_p$ operators take weakly holomorphic modular forms to weakly holomorphic modular forms, so they preserve modularity and holo\-morph\-icity in the upper half plane.  In this section, we study the effects that the $U_p$ and $V_p$ operators have on poles of weakly holomorphic modular forms.  Because the actions of $U_p$ and $V_p$ are easily described in terms of Fourier expansions, the effects on poles at $\infty$ are easily seen.  Hence, we concentrate on studying the poles at $0$.  A basic result is the following.

\begin{theorem}\cite[Corollary 4.2]{Doud-Jenkins}\label{swappoles}
Let $f$ be a weakly holomorphic modular form of even weight $k$ and level 1, and let $f_p=f|U_p$.  Then
$$p(pz)^{-k}f_p(-1/pz)=-f(z)+pf_p(pz)+p^kf(p^2z)$$
and $p(pz)^{-k}f_p(-1/pz)$ is modular of weight $k$ and level $p$.
\end{theorem}

We wish to extend this theorem to study $f|U_{p^2}$.  Using a proof similar to that found in \cite{Doud-Jenkins}, we obtain
\begin{theorem}\label{Up2} Let $f$ be a weakly holomorphic modular form of even weight $k$ and level 1, and let $f_{p^2}=f|U_{p^2}$ and $f_p=f|U_p$.  Then
$$p(pz)^{-k}f_{p^2}(-1/pz)=-f_p(z)+pf_{p^2}(pz)-p^{k-1}f(pz)+p^kf_p(p^2z)+p^{2k-1}f(p^3z).$$
In addition, $p(pz)^{-k}f_{p^2}(-1/pz)$ is modular of weight $k$ and level $p$.
\end{theorem}
\begin{proof}
We proceed as in the proof of \cite[Lemma 4.1]{Doud-Jenkins}.  For an integer $j$ with $1\leq j\leq p-1$, let $j'$ be the unique solution to $jj'\equiv 1\pmod p$ with $-(p-1)<j'<-1$, and set $b_j=\frac{jj'-1}p$.  We then write
\begin{align*}
pz^{-k}f_{p^2}(-1/z)&=pz^{-k}\frac1p\sum_{j=0}^{p-1}f_p\left(\frac{-1/z+j}{p}\right)\cr
&=z^{-k}\sum_{j=0}^{p-1}f_p\left(\frac{jz-1}{pz}\right)\cr
&=z^{-k}\sum_{j=1}^{p-1}f_p\left(\begin{pmatrix}j&b_j\cr
p&j'\end{pmatrix}\frac{z-j'}p\right)+z^{-k}f_p\left(\frac{-1}{pz}\right)\cr
&=z^{-k}\sum_{j=1}^{p-1}\left(p\left(\frac{z-j'}p\right)+j'\right)^kf_p\left(\frac{z-j'}p\right)+z^{-k}f_p\left(\frac{-1}{pz}\right)\cr
&=z^{-k}\sum_{j=1}^{p-1}z^kf_p\left(\frac{z-j'}p\right)\cr &\qquad
\qquad+z^{-k}\left(\frac{(pz)^k}p\left(-f(z)+pf_p(pz)+p^kf(p^2z)\right)\right)\cr
&=\sum_{j=0}^{p-1}f_p\left(\frac{z+j}p\right)-f_p(z/p)+p^{k-1}\left(-f(z)+pf_p(pz)+p^kf(p^2z)\right)\cr
&=pf_{p^2}(z)-f_p(z/p)-p^{k-1}f(z)+p^kf_p(pz)+p^{2k-1}f(p^2z)\cr
\end{align*}
The formula for $p(pz)^{-k}f_{p^2}(-1/pz)$ follows by replacing $z$ by $pz$ in this equation.  The statement about the level and the weight is immediate, once we notice that
$$-f_p(z)+pf_{p^2}(pz)-p^{k-1}f(pz)+p^kf_p(p^2z)+p^{2k-1}f(p^3z)$$
is equal to
$$p(f|T_p|T_p|V_p)-f|T_p-p^kf|V_p.$$
\end{proof}

We also note the trivial identity below that allows us to find the pole at 0 of $f|V_p$, where $f$ is a weakly holomorphic form of level 1.
\begin{proposition}\label{Vp}
Let $f$ be a weakly holomorphic modular form of weight $k$ and level 1.  Then
$$z^{-k}(f|V_p)(-1/pz)=f(z).$$
\end{proposition}

\section{Two-dissections}
As a first step in proving the divisibility results on the $a_k(m,n)$ we begin by reviewing Kolberg's method of two-dissections \cite{Kolberg}.  We do this in order to prove a key congruence that will be used in the proof of Lemma~\ref{ak2m2jn}.

Given a power series $$f=\sum_{n\in\Z} a(n)q^n,$$ we define the two-dissection $f=(f)_0+(f)_1$ by
$$(f)_0=f|U_2|V_2=\sum_{n\in\Z} a(2n)q^{2n}\text{ and } (f)_1=f-(f)_0=\sum_{n\in\Z} a(2n+1)q^{2n+1}.$$
Note that if $f$ is modular of weight $k$ and level $N$, then both $(f)_0$ and $(f)_1$ are modular of weight $k$ and level $4N$.  This two-dissection allows us to separate the coefficients of even and odd powers of $q$ in the expansion of $f$. If we wish to make a finer distinction than studying the coefficients of odd or even powers of $q$, we can (after dividing by $q$ in the case of $(f)_1$) replace $q^2$ by $q$ in the expansion of $(f)_0$ or $(f)_1$ and two-dissect the resulting function.  By two-dissecting a function repeatedly, we can isolate powers of $q$ with exponent in a fixed residue class modulo a power of 2.

To allow us to easily compute two-dissections, we define (following Kolberg) the following functions:
$$\varphi(q)=\prod_{k=1}^\infty(1-q^k),\qquad Q(q)=\prod_{k=1}^\infty(1+q^k)=\frac{\varphi(q^2)}{\varphi(q)},$$
$$R(q)=Q(q^2),S(q)=R(q^2),T(q)=S(q^2).$$
Note that although the expansion of $Q$ contains odd powers of $q$, the expansions of $R$, $S$, and $T$ contain only even powers of $q$.  In addition, under the transformation taking $q^2$ to $q$, we see that $R$ maps to $Q$, $S$ maps to $R$, and $T$ maps to $S$.

Kolberg showed \cite[p. 5]{Kolberg}, that we can write
$$Q^2=R(\cos\alpha+i\sin\alpha)$$
with $\cos(\alpha)=R^2S^3T^{-2}$ and $\sin(\alpha)=-2iqR^2ST^2$.  One then checks that $(Q^2)_0=R\cos\alpha$ and $(Q^2)_1=iR\sin(\alpha)$.  We then compute the two-dissection of $Q^{2k}$ for any integer value of $k$ by setting
$$Q^{2k}=R^k(\cos(k\alpha)+i\sin(k\alpha))$$
and using trigonometric identities to compute $\cos(k\alpha)$ and $\sin(k\alpha)$ in terms of $Q, R, S,$ and $T$.  For instance
$$\sin(2\alpha)=2\sin(\alpha)\cos(\alpha)=-4iqR^4S^4.$$
Other basic identities used by Kolberg are
$$\cos(2\alpha)=R^8S^{-4}\quad\text{ and }\quad \sin(4\alpha)=-8iqR^{12}.$$

We have written a computer program to compute two-dissections of
$q$-series. Since our purpose is to compute congruences modulo
powers of two, we are interested in computing two-dissections with
few terms modulo given powers of two.

For odd dissections of powers of $Q$, the program needs to compute
$\sin(k\alpha)$ in terms of Kolberg's functions.  It does this by
using Chebyshev polynomials  to write $\sin(k\alpha)$ as
$\sin(\alpha)$ times a polynomial in $\cos(\alpha)$ when $k$ is odd,
as  $\sin(2\alpha)$ times a polynomial in $\cos(2\alpha)$ if $k$ is
even, or as $\sin(4\alpha)$ times a polynomial in
$\cos(4\alpha)=1-2\sin^2(2\alpha)$ if $k$ is divisible by 4.

For even dissections of powers of $Q$, the program must compute
$\cos(k\alpha)$ in terms of Kolberg's functions.  To do this for odd
values of $k$, it uses Chebyshev polynomials to write
$\cos(k\alpha)$ as a polynomial in $\cos(\alpha)$.  If $k$ is
divisible by $4$, it computes $k$ as a polynomial in
$\sin^2(2\alpha)$.  Otherwise it uses the identity
$\cos(k\alpha)=1-2\sin^2((k/2)\alpha)$ to reduce the problem to
computing a sine function, as described above.

To illustrate this, we demonstrate the two-dissection of $Q^{16}$.
We note that we can write
$\cos(8\alpha)=1-8\sin^2(2\alpha)+8\sin^4(2\alpha)$, so we have
\begin{align*}
(Q^{16})_0&=R^8\cos(8\alpha)\cr
&=R^8(1-8\sin^2(2\alpha)+8\sin^4(2\alpha))\cr
&=R^8+2^7qR^{16}S^8+2^{11}q^4R^{24}S^{16},
\end{align*}
where we have substituted $-4iqR^4S^4$ for $\sin(2\alpha)$.  This shows that $(Q^{16})_0$ is congruent to $R^8$ modulo $2^7$.

Similarly, we compute the odd two-dissection of $Q^{16}$ using the identity
$$\sin(8\alpha)=2\sin(4\alpha)\cos(4\alpha)=2\sin(4\alpha)(1-2\sin^2(2\alpha)).$$
We then substitute in for $\sin(4\alpha)$ and $\sin(2\alpha)$, obtaining
$$(Q^{16})_1=R^8i\sin(8\alpha)=R^8i(2\sin(4\alpha)(1-2\sin^2(2\alpha)))=16qR^{20}+512q^3R^{28}S^8.$$
This shows that $(Q^{16})_1$ is congruent to $16qR^{20}$ modulo $2^9$, and also that it is congruent to $0$ modulo $2^4$.

To dissect arbitrary functions written in terms of $Q$, $R$, $S$,
$T$, $\varphi(q^2)$, $\varphi(q^4)$, and $\varphi(q^8)$, we see that
it suffices to be able to two-dissect powers of $Q$, since only $Q$
contains terms with odd powers of $q$.

We note that the computer program uses no approximate techniques; in all cases, it manipulates trigonometric identities to obtain proven dissections. Using this program, we prove the following result:

\begin{lemma}\label{dissection}\label{ak2m4} For $k\in\{12,16,18,20,22,26\}$ and  $m$ positive and odd, we have
$$a_k(2m,4)\equiv 0 \pmod {2^{\rho+\gamma}}.$$
\end{lemma}

\begin{proof}
By duality, this is equivalent to showing that
$$a_{2-k}(4,2m)\equiv 0 \pmod{ 2^{\rho+\gamma}}$$
for each $k$. In other words, we wish to show that $a_{2-k}(4,j)\equiv 0\pmod{2^{\rho+\gamma}}$ for all $j\equiv 2\pmod 4$.

To prove these results, we set $g(q^2)=(f_{2-k,4})_0=\sum_{n=0}^\infty a_{2-k}(4,2n)q^{2n}$.  Then
$$g(q)=\sum_{n=0}^\infty a_{2-k}(4,2n)q^n$$
and the odd two-dissection of $g(q)$ is
$$(g(q))_1=\sum_{n=0}^\infty a_{2-k}(4,2(2n+1))q^{2n+1}=\sum_{n=0}^\infty a_{2-k}(4,4n+2)q^{2n+1},$$
so we see that we want $(g(q))_1\equiv 0\pmod {2^{\rho+\gamma}}$.

To illustrate this, we will look at the $k=16$ case.  We are
obtaining a result modulo $2^{\rho+\gamma}=2^{11}$.  We use the
facts that \cite{Kolberg,Ono}
$$\Delta(z)=q\varphi(q)^{24},\qquad E_4(z)=\varphi(q^2)^8(Q^{-16}+2^8qQ^8),$$
\begin{align*}E_6(z)=\varphi(q^2)^{12}&Q^{-24}-2^5\cdot3\cdot5q\varphi(q^2)^{12}-2^9
\cdot3\cdot11q^2Q^8\varphi(q^2)^4\varphi(q^4)^8\cr&+2^{13}q^3\varphi(q^2)^{-12}\varphi(q^4)^{24},\end{align*}
and $$j(z)=E_4(z)^3/\Delta(z)=q^{-1}Q^{-24}+2^8\cdot3+2^{16}\cdot3qQ^{24}+2^{24}q^2Q^{48},$$
to write $f_{-14,4}(z)=E_4(z)E_6(z)\Delta^{-2}(z)(j(z)^2-1272j(z)+192600)$ modulo $2^{12}$ as
$$\frac{\varphi(q^2)^{8}+40q\varphi(q^2)^8Q^{24}+3584q^2\varphi(q^4)^8Q^{32}+
344q^2\varphi(q^2)^8Q^{48}+768q^3\varphi(q^2)^8Q^{72}}{q^4Q^{40}\varphi(q^2)^{36}}.$$
Note that to get equality, there would be many additional terms; we
have eliminated all terms that are multiples of $2^{12}$.

We now use the program to two-dissect $f_{-14,4}$ by two-dissecting each of the powers of $Q$ appearing in this expansion, obtaining
$$g(q^2)\equiv\left(f_{-14,4}
\right)_0\equiv \frac{2816 R^{40} S^8}{\varphi(q^4)^{28}}+\frac{R^8}{\varphi(q^4)^{28} q^4}+\frac{3288 R^{32}}{\varphi(q^4)^{28} q^2}+\frac{800 R^{16} S^8}{\varphi(q^4)^{28} q^2}\pmod{ 2^{12}}.$$
Replacing $q^2$ by $q$ yields
$$g(q)\equiv\frac{2816 Q^{40} R^8}{\varphi(q^2)^{28}}+\frac{Q^8}{\varphi(q^2)^{28} q^2}+\frac{3288 Q^{32}}{\varphi(q^2)^{28} q}+\frac{800 Q^{16} R^8}{\varphi(q^2)^{28} q}\pmod{ 2^{12}}.$$
Finally, we perform the odd dissection of $g(q)$ to get the desired result:
$$\left(g(q)\right)_1\equiv\frac{2048 q R^{40}}{\varphi(q^2)^{28}}\pmod{2^{12}}.$$
Therefore, $a_{16}(2m,4)$ is divisible by $2^{11}$.

The proofs for the other weights follow similarly, by writing $f_{2-k,4}(z)$ in terms of $E_4$, $E_6$, $\Delta$, and $j$, and two-dissecting the resulting expression twice.  In each case, we find that the coefficients of powers of $q$ with exponent congruent to 2 modulo 4 are all divisible by $2^{\rho+\gamma}$.

\end{proof}

\section{Coefficients $a_k(2^am,2^bn)$ with $a>b$}
\subsection{Special bases for $M_k(2)$}
Writing holomorphic modular forms in terms of special basis elements will often help us to prove congruences on their coefficients (see Lemma~\ref{ak0n}).  Standard dimension formulas \cite[Proposition 6.1]{Stein} yield the following formula for the dimension of $M_k(2)$.
\begin{theorem}
For $k\geq 4$ even,
$$\dim(M_k(2))=1+\left\lfloor\frac{k}4\right\rfloor.$$
\end{theorem}
Hence, in order to find a basis for $M_k(2)$, we need only find $1+\left\lfloor\frac{k}4\right\rfloor$ linearly independent modular forms in $M_k(2)$.  For $k\in\{12,16,18,20,22,26\}$ we do this as follows.
\begin{lemma}\label{level2basis} For $k\in\{12,16,18,20,22,26\}$, the modular forms listed in Table~\ref{basis} form a basis of $M_k(2)$.
\begin{table}[h]
\begin{center}
\begin{tabular}{|c|l|}
\hline $k$&{\rm Basis Elements}\cr \hline $12$&
$E_4(2z)^3,\Delta(z),\Delta(2z),S_4(z)^3$\cr
$16$&$E_4(2z)^4,\Delta_{16},\Delta_{16}(2z),E_4(z)S_4(z)^3,S_4(z)^4$\cr
$18$&$E_4(2z)^3E_6(2z),\Delta_{18},\Delta_{18}(2z),S_4^3E_6,S_4^3S_6$\cr
$20$&$E_4(2z)^5,\Delta_{20},\Delta_{20}(2z),S_4^3E_4^2,S_4^4E_4,S_4^5$\cr
$22$&$E_4(2z)^4E_6(2z),\Delta_{22},\Delta_{22}(2z),S_4^3E_4E_6,S_4^4E_6,S_4^4S_6$\cr
$26$&$E_4(2z)^5E_6(2z),\Delta_{26},\Delta_{26}(2z),S_4^3E_4^2E_6,S_4^4E_4E_6,S_4^5E_6,S_4^5S_6$\cr
\hline
\end{tabular}
\caption{Bases for $M_k(2)$}\label{basis}
\end{center}
\end{table}All of the forms in these bases have integer coefficients and leading coefficient $1$, and the $i$th form in each basis vanishes to order $i-1$ at $\infty$.
\end{lemma}
\begin{proof} One easily checks that each of the listed forms is in the appropriate $M_k(2)$, has integer coefficients, and vanishes to the indicated order at $\infty$.  For a given $k$, each form has a different order of vanishing, which implies that they are linearly independent, and since there are $1+\left\lfloor\frac k4\right\rfloor$ of them, they form a basis.
\end{proof}

We now prove a result on the divisibility of the coefficients of the holomorphic modular form $f_{k,0}$.
\begin{lemma}\label{ak0n} For $k\in\{12,16,18,20,22,26\}$, and for $n$ odd and positive, $a_k(0,n)\equiv 0\pmod {2^\eta}$.
\end{lemma}
\begin{proof} Note that $f_{k,0}\in M_k(2)$, hence it can be written in terms of the basis described in Lemma~\ref{level2basis}.  For $k=12$, we obtain
$$f_{12,0}(z)=E_4(2z)^3+2^8765\Delta(2z)+2^{12}4095S_4(z)^3.$$
Note that the only form with nonzero odd coefficients is multiplied by $2^{12}$ so that all odd coefficients of $f_{12,0}$ are divisible by $2^{12}$.  The proof for all other weights is similar--when we write $f_{k,0}$ in terms of the basis, all forms with odd coefficients are multiplied by $2^\eta$.\end{proof}
Although it is not used anywhere else in the proof, we include the following result for completeness.
\begin{lemma}\label{ak02in} For $k\in\{12,16,18,20,22,26\}$, for $n$ odd and positive, and for $b>0$, $a_k(0,2^bn)\equiv 0\pmod {2^\omega}$
\end{lemma}
\begin{proof}Each $f_{k,0}$ is a holomorphic modular form of level 1, with constant term 1, and
with the coefficient of $q$ equal to 0.  We may thus write it as
$$f_{k,0}=E_4(z)^rE_6(z)^s-C\Delta_k(z)$$
where $4r+6s=k$, $0\leq s\leq 1$, and $C$ is the coefficient of $q$
in $E_4(z)^aE_6(z)^b$. Using the facts that $E_4(z)\equiv 1\pmod
{2^4}$ and $E_6(z)\equiv 1\pmod {2^3}$, one easily checks that
$E_4(z)^rE_6(z)^s\equiv 1\pmod{ 2^\omega}$, and the lemma follows.
\end{proof}


Using the above result, we now prove the following two results--the first about coefficients of odd powers of $q$ in specific forms, and the second generalizing this result to odd powers of $q$ in all forms $f_{k,m}$ with $m$ even.
\begin{lemma}
\label{ak2n} For  $k\in\{12,16,18,20,22,26\}$ and $n$ positive and odd, $a_k(2,n)\equiv 0\pmod {2^\nu}.$
\end{lemma}
\begin{proof}
We examine the function
$G(z)=f_{k,2}(z)-(f_{k,1}|V_2)(z)-a_k(2,2)\Delta_k(2z)\in M_k^!(2)$.
This function is holomorphic at $\infty$ and in the upper half
plane. Using Proposition~\ref{Vp}, together with the modularity of
$f_{k,2}$ and $\Delta_k$, we easily compute the expansion of $G(z)$
at 0 to be
$$2(2z)^{-k}G(-1/(2z))=2q^{-4}-2^{1-k}q^{-1}+O(1).$$
Hence, we may write
$$2(2z)^{-k}G(-1/(2z))=\sum_{j=1}^4C_j\psi^{j+1}(z)\Delta_k(z)+F'(z)$$
where $F'(z)$ is holomorphic at $\infty$.  From the expansion given
above, we see that $2^{k-1}C_j\in\mathbb Z$ for all $j$. Now
replacing $z$ by $-1/(2z)$ and dividing by $2z^k$, we obtain
$$G(z)=\sum_{j=1}^4\frac{1}{2}C_j2^{12(j+1)}\Phi(z)^{j+1}2^k\Delta_k(2z)+\frac{1}{2}z^{-k}F'(-1/(2z)).$$
We see from this that $\frac{1}{2}z^{-k}F'(-1/(2z))$ is a
holomorphic modular function of weight $k$ and level 2, which we
will denote by $F(z)$. We note that the coefficients
$\frac{1}{2}C_j2^{12j+12+k}$ are all divisible by 2 to the power of
$12j+12+k-(k-1)-1=12j+12$.  Hence, $F(z)$ actually has integer
coefficients, and
$$G(z)\equiv F(z)\pmod {2^{24}}.$$
Since $F(z)$ is a linear combination of the basis elements listed
above, we see that if the first $\dim(M_k(2))$ coefficients of
$G(z)$ are divisible by $2^\nu$, then the same is true for the first
coefficients of $F(z)$, and hence for all coefficients of $F(z)$ and
$G(z)$.  We have computed $G(z)$ for each
$k\in\{12,16,18,20,22,26\}$ and verified that all of the first
$\dim(M_k(2))$ coefficients are divisible by $2^\nu$.  Hence, the
lemma is proved.
\end{proof}

\begin{proposition}\label{odd-coefficients} For $k\in\{12,16,18,20,22,26\}$,
$m,n$ positive and odd, and $a>0$, $$a_k(2^am,n)\equiv 0\pmod {2^\nu}.$$\end{proposition}
\begin{proof} We will show that $f_{2-k,n}|U_2\equiv a_{2-k}(n, 0)\pmod {2^\nu}$.  This
immediately implies that for all $a>0$, the coefficient $a_{2-k}(n,2^am)\equiv 0\pmod{2^\nu}$, and we are done by duality.

Now $f_{2-k,n}|U_2$ is holomorphic at $\infty$, and using
Theorem~\ref{swappoles}, we see that it has a pole of order $4n$ at
0.  In fact, choosing values of $C_j$ to cancel out all negative
powers of $q$, we see that
$$2(2z)^{-(2-k)}(f_{2-k,n}|U_2)(-1/(2z))=2^{2-k}q^{-4n}+\cdots=\sum_{j=0}^{4n}C_j\psi(z)^j\theta_{2-k}(z),$$
where equality follows from the fact that there are no nonzero holomorphic modular forms of negative weight.
Note that $A_j=2^{k-2}C_j$ is an integer.

We then have that
$$f_{2-k,n}|U_2=\pm\sum_{j=0}^{4n}2^{2-k}A_j2^{12j-1}2^{\mu_{2-k}}\alpha_{2-k}(z)\Phi(z)^j,$$
with the sign depending on $k$.

Reducing this modulo $2^{16}$, we find that all but the first two terms vanish.  We then find that for some integer $C$,
$$f_{2-k,n}|U_2\equiv a_{2-k}(n,0)\alpha_{2-k}(z)+C\alpha_{2-k}(z)\Phi(z)\pmod{2^{16}}.$$
Since $a_{2-k}(n,0)\equiv0\pmod {2^\eta}$ (by Lemma~\ref{ak0n}), and $\alpha_{2-k}\equiv 1\pmod {2^{\xi_{2-k}}}$, we see that $C\equiv a_{2-k}(n,2)\pmod{2^{\eta+\xi_{2-k}}}$.  Hence, as long as $\eta+\xi_{2-k}\geq\nu$, we are done, since $a_{2-k}(n,2)\equiv 0\pmod {2^{\nu}}$ by Lemma~\ref{ak2n} and duality. One checks easily for each value of $k$ that $\eta+\xi_{2-k}=\nu$.\end{proof}

Next we prove a congruence on coefficients of $q^2$ in certain forms, which will allow us to complete the proof of Proposition~\ref{lower}, which gives the desired divisibility for all $a_k(2^im,2^jn)$ with $i>j$.

\begin{lemma}\label{ak2in2} Let $k\in\{12,16,18,20,22,26\}$. For $a>1$ and $m$ positive and odd, we have $a_k(2^am,2)\equiv 0\pmod{2^\chi}$.\end{lemma}

\begin{proof}
By duality, it is enough to show that $a_{2-k}(2,2^am)\equiv 0\pmod {2^\chi}$.  We will do this by showing that $f_{2-k,2}|U_4\equiv a_{2-k}(2,0)\pmod{2^\chi}$.  This will imply that all coefficients of $q^n$ in $f_{2-k,2}$ with $4|n$ and $n>0$ will be divisible by $2^\chi$.

We begin by noting that by Theorem~\ref{Up2}, $f_{2-k,2}|U_4$ is a modular form of weight $2-k$ and level 2, holomorphic at infinity and on the upper half plane, with a pole of order 16 at 0.  This implies that it can be written as a linear combination
$$f_{2-k,2}|U_4=\sum_{i=0}^{16-\delta}C_i\Phi^i\alpha_{2-k},$$
where $\delta$ is the order of the pole at 0 of $\alpha_{2-k}$.  We could use Theorem~\ref{Up2} to directly compute the coefficients $C_i$, but instead use the fact that the $\Phi^i\alpha_{2-k}$ are linearly independent to compute the expansion directly (by cancelling out the constant term, then the $q$ term, etc.).  The expansion for $k=12$ is then computed to be
\begin{align*}
f_{-10,2}|U_4&=
-2^{4}12285\alpha_{-10}-2^{13}560584703\Phi^1\alpha_{-10}-2^{32}81073083\Phi^2\alpha_{-10}\cr&\ \ \ \
-2^{41}1530945507\Phi^3\alpha_{-10}-2^{52}2152441109\Phi^4\alpha_{-10}-2^{61}5180299059\Phi^5\alpha_{-10}\cr&\ \ \ \
-2^{72}1612138839\Phi^6\alpha_{-10}-2^{81}1156390619\Phi^7\alpha_{-10}-2^{93}63543825\Phi^8\alpha_{-10}\cr&\ \ \ \
-2^{102}17635779\Phi^9\alpha_{-10}-2^{117}48367\Phi^{10}\alpha_{-10}-2^{126}5199\Phi^{11}\alpha_{-10}\cr&\ \ \ \
-2^{138}39\Phi^{12}\alpha_{-10}-2^{147}\Phi^{13}\alpha_{-10}.\cr
\end{align*}
All terms except the first vanish modulo $2^7$, and, since $\alpha_{-10}\equiv 1\pmod {2^3}$, we see that
$$f_{-10,2}|U_4\equiv -2^412285\pmod {2^7}.$$

The proof for the other weights proceeds similarly, computing the expansion of $f_{2-k,2}|U_4$ and noting that all terms except the first vanish modulo $2^\chi$ and that the first term is a constant modulo $2^\chi$.
\end{proof}

\begin{proposition}\label{lower} Let $k\in\{12,16,18,20,22,26\}$. For $a> b\geq0$, and for $m,n$ positive and odd, we have
$$a_k(2^am,2^bn)\equiv 0\pmod {2^\chi}.$$\end{proposition}
\begin{proof} We proceed by induction on $b$.  The case $b=0$ is just Proposition~\ref{odd-coefficients}.

We assume (by way of induction) that the theorem is true for all $b\leq B$.  Assume that $a>B+1$.  Then by Lemma~\ref{Hecke}, with  $m$ and $n$ replaced respectively by $2^am$ and $2^Bn$, we obtain
\begin{align*}a_k(2^{a-1}m,2^Bn)+a_k(2^am,2)&\tau_k(2^Bn)+2^{k-1}a_k(2^{a+1}m,2^Bn)\cr&=a_k(2^am,2^{B+1}n)+2^{k-1}a_k(2^am,2^{B-1}n).\end{align*}
Since $k-1\geq \chi$, this reduces to
 $$a_k(2^{a-1}m,2^Bn)+a_k(2^am,2)\tau_k(2^Bn)=a_k(2^am,2^{B+1}n) \pmod{2^\chi}.$$
However, by our induction hypothesis, $a_k(2^{a-1}m,2^Bn)$ vanishes modulo $2^\chi$, and by Lemma~\ref{ak2in2},  $a_k(2^am,2)$ vanishes modulo $2^\chi$.  Hence, the theorem is proven.\end{proof}

\section{Coefficients $a_k(2^am,2^bn)$ with $a<b$}
We now prove the theorem for the coefficients $a_k(2^am,2^bn)$ for which $a<b$.  We will prove this by induction on the difference $b-a$.  To begin, we deal with the cases where $b-a<3$ (each of which is proved by induction on $a$).  In order to prove these cases, we first need the following divisibility result for coefficients of $q^2$.

\begin{lemma}\label{akm2}
Let $k\in\{12,16,18,20,22,26\}$.  Then for odd $m$, $a_k(m,2)$ is divisible by $2^\gamma$.
\end{lemma}
\begin{proof} We note that $v_2(a_k(m,2))=v_2(a_{2-k}(2,m))$ by duality.  For $k\neq 22$, we note (by explicit computation and comparing poles) that $f_{2-k,2}E_4^3=f_{2-k+12,2}+C_kf_{2-k+12,1}$, with $C_{12}=744$, $C_{16}=504$, $C_{18}=1248$, $C_{20}=264$, and $C_{26}=768$.  Since $E_4^3\equiv 1\pmod {2^4}$, we see that
\begin{align*}v_2(a_k(m,2))&=v_2(a_{2-k}(2,m))\cr
&\geq\min(v_2(a_{2-k+12}(2,m)),4,v_2(C_k))\cr
&=\min(v_2(a_{k-12}(m,2),4,v_2(C_k))\cr
&=\min(4,v_2(C_k)),\end{align*}
where the last equality arises from the fact that $v_2(a_{k-12}(m,2))\geq 7$ by \cite[Theorem 1.3]{Doud-Jenkins} and \cite[Theorem 2.1]{Griffin}. Since $\min(4,v_2(C_k))=\gamma$, we are done.

For $k=22$, we do a similar computation, using that $E_4^4\equiv 1\pmod{2^5}$:
$$f_{-20,2}E_4^4=f_{-4,2}+ 1248f_{-4,1}.$$
Hence for $m$ odd, $$a_{22}(m,2)=-a_{-20}(2,m)\equiv -a_{-4}(2,m)=a_6(m,2)\equiv 0\pmod{2^5}$$ by duality and \cite[Theorem 1.3]{Doud-Jenkins}.
\end{proof}

We now prove the cases $b-a=1$ and $b-a=2$.
\begin{lemma}\label{j-i=1} Let $k\in\{12,16,18,20,22,26\}$.  For $a\geq 0$ and $m,n>0$ odd, we have
$$a_k(2^am,2^{a+1}n)\equiv 0\pmod{2^\gamma}.$$
\end{lemma}

\begin{proof}
Applying Lemma~\ref{Hecke}, we obtain
$$a_k(m,2)\tau_k(n)+2^{k-1}a_k(2m,n)=a_k(m,2n).$$
Now $k-1>\gamma$ and by Lemma~\ref{akm2}, $a_k(m,2)$ is divisible by $2^\gamma$, so we see that $a_k(m,2n)\equiv 0\pmod {2^\gamma}$, and  the Lemma is proved for $a=0$.

We now assume that $A>0$ and that the Lemma is true for $a=A-1$.  Applying Lemma~\ref{Hecke} with  $m,n$ replaced by $2^Am$ and $2^An$,
$$a_k(2^{A-1}m,2^An)+a_k(2^Am,2)\tau_k(2^An)\equiv a_k(2^Am,2^{A+1}n)\pmod {2^\gamma}.$$
By Lemma~\ref{tau}, $\tau_k(2^An)\equiv 0\pmod {2^\gamma}$, so by the induction hypothesis we see that  $a_k(2^Am,2^{A+1}n)\equiv 0\pmod{2^\gamma}$.\end{proof}

\begin{lemma}\label{j-i=2} For $k\in\{12,16,18,20,22,26\}$,  $a\geq 0$ and $m,n\in\Z$ positive and odd,
$$a_k(2^am,2^{a+2}n)\equiv 0\pmod {2^{2\gamma}}.$$
\end{lemma}
\begin{proof}
Replacing $n$ by $2n$ in Lemma~\ref{Hecke}, we obtain
$$a_k(m,4n)\equiv a_k(m,2)\tau_k(2n)\pmod {2^{k-1}}.$$
By Lemmas~\ref{tau} and \ref{akm2}, and the fact that $2\gamma<k-1$, we see that $a_k(m,4n)\equiv 0\pmod {2^{2\gamma}}$.

Letting $A>0$ and assuming the theorem for $a=A-1$, we apply Lemma~\ref{Hecke} with  $m,n$ replaced by $2^{A}m$ and $2^{A+1}n$, to obtain
$$a_k(2^Am,2^{A+2}n)\equiv a_k(2^{A-1}m,2^{A+1}n)+a_k(2^{A}m,2)\tau_k(2^{A+1}n)\pmod {2^{k-1}}.$$
Lemma~\ref{tau} and the induction hypothesis then imply that $a_k(2^Am,2^{A+2}n)\equiv 0\pmod {2^{2\gamma}}$.
\end{proof}

Next, we prove a weaker divisibility than our final result, from which we will derive the stronger divisibility that we need.

\begin{lemma} \label{i<j} Let $k\in\{12,16,18,20,22,26\}$. For $0\leq a<b$ and $m,n$ positive and odd, we have that
$$a_k(2^am,2^bn)\equiv 0\pmod{2^{\gamma(b-a)}}.$$
\end{lemma}
\begin{proof} Note that the theorem has already been proven for $b-a=1$ (Lemma~\ref{j-i=1}) and for $b-a=2$ (Lemma~\ref{j-i=2}).  Assume, by way of induction, that the theorem is true for $0<b-a<N$.  Then applying Lemma~\ref{Hecke} with $m$ and $n$ replaced by $2^am$ and $2^{a+N-1}n$, we obtain
\begin{align*}a_k(2^{a-1}m,2^{a+N-1}n)+a_k(2^am,2)&\tau_k(2^{a+N-1}n)+2^{k-1}a_k(2^{a+1}m,2^{a+N-1}n)\cr&=a_k(2^am,2^{a+N}n)+2^{k-1}a_k(2^am,2^{a+N-2}n)\end{align*}

We note that $a_k(2^am,2)\tau_k(2^{a+N-1}n)\equiv 0\pmod {2^{\gamma N}}$ (using Lemma~\ref{tau} if $a>0$, and using Lemmas~\ref{tau} and \ref{akm2} if $a=0$).  Further, by our induction hypothesis, $2^{k-1}a_k(2^{a+1}m,2^{a+N-1}n)$ and $2^{k-1}a_k(2^am,2^{a+N-2}n)$ are both divisible by $2^{k-1+\gamma(N-2)}$, which (since $k-1>2\gamma$) is divisible by $2^{\gamma N}$. Hence, we see that
$$a_k(2^am,2^{a+N}n)\equiv a_k(2^{a-1}m,2^{a+N-1}n)\pmod {2^{\gamma N}}.$$
For $a=0$, the right hand side of this congruence is 0.  Hence, by a simple induction on $a$, the left hand side is 0 modulo $2^{\gamma N}$ for all nonnegative $a$.  This completes the induction.
\end{proof}

For coefficients $a_k(2^am,2^bn)$ with $b>a$ and $a=0$, the congruence in Lemma~\ref{i<j} is our final result.  However, for $a>0$, we can derive a better congruence by bootstrapping from this result.  The proof will again be by induction on $b-a$, but we need an additional divisibility result on the coefficients $a_k(2m,2^bn)$ with $b\geq2$ to complete the induction.

\begin{lemma}\label{ak2m2jn} Let $k\in\{12,16,18,20,22,26\}$. For $b\geq 2$ and $m,n$ positive and odd,
$$a_k(2m,2^bn)\equiv 0\pmod {2^{\rho+\gamma(b-1)}}.$$
\end{lemma}
\begin{proof}
Applying Lemma~\ref{Hecke2} with $m$ and $n$ replaced by $2m$ and $2^{B-2}n$, we obtain
\begin{align*}2^{2k-2}a_k(8m,2^{B-2}n)&+(1-\delta_{2^{B-2}n,2})2^{k-1}a_k(2m,2^{B-2}n)+a_k(2m,4)\tau_k(2^{B-2}n)\cr&=2^{2k-2}a_k(2m,2^{B-4}n)+a_k(2m,2^Bn).\end{align*}
For $B=2$, we note that $\rho+(B-1)\gamma=\rho+\gamma\leq k-1$.  For $3\leq B\leq 6$, we note that $\rho+(B-1)\gamma\leq\rho+5\gamma\leq 2k-2$ and that $(1-\delta_{2^{B-2}n,2})=0$.  In both cases, the formula above reduces to
$$a_k(2m,2^Bn)\equiv a_k(2m,4)\tau_k(2^{B-2}n)\pmod{2^{\rho+(B-1)\gamma}}.$$
Now, by Lemma~\ref{ak2m4}, we know that $a_k(2m,4)$ is divisible by $2^{\rho+\gamma}$, and by Lemma~\ref{tau}, we know that $\tau_k(2^{B-2}n)$ is divisible by $2^{\gamma(B-2)}$.  Hence, for $2\leq b\leq 6$ we have that $a_k(2m,2^bn)\equiv 0\pmod {2^{\rho+(b-1)\gamma}}$.

We will now assume that $B>6$, and that the theorem is true for all $b$ with $2\leq b<B$.  Now, since $2k-2>4\gamma$, the induction hypothesis implies that $2^{2k-2}a_k(2m,2^{B-4}n)$ vanishes modulo $2^{\rho+(B-1)\gamma}$. Lemma~\ref{i<j} and the fact that $2k-2>\rho+4\gamma$ imply that $2^{2k-2}a_k(8m,2^{B-2}n)$ vanishes modulo $2^{\rho+(B-1)\gamma}$.  By Lemmas~\ref{tau} and \ref{ak2m4}, $a_k(2m,4)\tau_k(2^{B-2}n)$ vanishes modulo $2^{\rho+(B-1)\gamma}$.  Hence, $a_k(2m,2^Bn)\equiv 0\pmod {2^{\rho+(B-1)\gamma}}$, and by induction the theorem is true for all $b\geq 2$.
\end{proof}

The proof for $b-a=1$ proceeds similarly to the proof of Lemma~\ref{j-i=1}.
\begin{lemma}\label{ak2im2i+1n} Let $k\in\{12,16,18,20,22,26\}$. For $a\geq 1$ and $m,n$ positive and odd,
$$a_k(2^am,2^{a+1}n)\equiv 0\pmod {2^{\rho+\gamma}}.$$
\end{lemma}
\begin{proof}
By Lemma~\ref{ak2m2jn} we have that $a_k(2m,4n)\equiv 0\pmod{2^{\rho+\gamma}}$.  Hence, the theorem is true for $a=1$.

Assume now that $A>1$, and that the theorem is true for all $1\leq a<A$.  By Lemma~\ref{Hecke} with $m,n$ replaced by $2^Am$ and $2^An$, we find
$$a_k(2^Am,2^{A+1}n)\equiv a_k(2^{A-1}m,2^An)+a_k(2^Am,2)\tau_k(2^An)\pmod{2^{k-1}}.$$
We know that $\rho+\gamma<k-1$, that for $A>1$, $2^{2\gamma}|\tau_k(2^An)$ by Lemma~\ref{tau}, that $2^\chi|a_k(2^Am,2)$ by Lemma~\ref{ak2in2},  that $2^{\rho+\gamma}|a_k(2^{A-1}m,2^An)$ by the induction hypothesis, and that $\chi+\gamma>\rho$. Hence, $a_k(2^Am,2^{A-1}n)\equiv 0\pmod{2^{\rho+\gamma}}$, completing our induction.
\end{proof}

In order to complete the case where $b-a=2$ (Lemma~\ref{i-i+2}), we need the following result on coefficients for which $a=b$.
\begin{lemma}\label{difference} Let $k\in\{12,16,18,20,22,26\}$. For $b\geq 1$ and $m,n$ odd, $$a_k(2^{b+1}m,2^{b+1}n)\equiv a_k(2^bm,2^bn)\pmod{2^{\chi+\gamma}}.$$ \end{lemma}

\begin{proof}
Applying Lemma~\ref{Hecke} with  $m,n$ replaced by $2^{b+1}m$ and $2^bn$, we obtain
\begin{align*}a_k(2^bm,2^bn)+a_k(2^{b+1}m,2)&\tau_k(2^bn)+2^{k-1}a_k(2^{b+2}m,2^bn)\cr&=a_k(2^{b+1}m,2^{b+1}n)+2^{k-1}a_k(2^{b+1}m,2^{b-1}n).\end{align*}
Since $k-1>\chi+\gamma$, and by Lemma~\ref{tau} and Proposition~\ref{lower} we know that $a_k(2^{b+1}m,2)\tau_k(2^bn)$ vanishes modulo $2^{\chi+\gamma}$, the result follows.
\end{proof}
\begin{lemma}\label{i-i+2} Let $k\in\{12,16,18,20,22,26\}$. For $a\geq 1$ and $m,n$ positive and odd,
$$a_k(2^am,2^{a+2}n)\equiv 0\pmod{2^{\rho+2\gamma}}.$$
\end{lemma}
\begin{proof}
By Lemma~\ref{ak2m2jn}, we see that the lemma is true for $a=1$. We will now assume that $A>1$ and that the lemma is true for all $a$ with $1\leq a<A$.  By Lemma~\ref{Hecke} with $m,n$ replaced by $2^Am$ and $2^{A+1}n$ we obtain
\begin{align*}a_k(2^Am,2^{A+2}n)=a_k(2^{A-1}m&,2^{A+1}n)+a_k(2^Am,2)\tau_k(2^{A+1}n)\cr&+2^{k-1}(a_k(2^{A+1}m,2^{A+1}n)-a_k(2^Am,2^An)).\end{align*}

By Lemma~\ref{difference} we know that $2^{k-1}(a_k(2^{A+1}m,2^{A+1}n)-a_k(2^Am,2^An))$ vanishes modulo $2^{k-1+\chi+\gamma}$. Since $k-1+\chi+\gamma>\rho+2\gamma$, we see that this term vanishes modulo $2^{\rho+2\gamma}$.
By our induction hypothesis, $a_k(2^{A-1}m,2^{A+1}n)$ vanishes modulo $2^{\rho+2\gamma}$.
Finally, by Lemmas~\ref{ak2in2} and \ref{tau}, we see that
$$a_k(2^Am,2)\tau_k(2^{A+1}n)$$
vanishes modulo $2^{\rho+2\gamma}$, since  $\chi+(A+1)\gamma=(\chi+\gamma)+A\gamma\geq\rho+2\gamma$.
\end{proof}

We now conclude by induction on $b-a$, using Lemmas ~\ref{ak2im2i+1n} and \ref{i-i+2} as base cases.

\begin{proposition}\label{upper} Let $k\in\{12,16,18,20,22,26\}$. For $b>a\geq 1$ and $m, n$ positive and odd, we have that
$$a_k(2^am,2^bn)\equiv 0\pmod{2^{\rho+\gamma(b-a)}}.$$
\end{proposition}

\begin{proof} The result is true for $b-a=1$ and $b-a=2$ by Lemmas~\ref{ak2im2i+1n} and \ref{i-i+2}.  In addition, by Lemma~\ref{ak2m2jn} it is true for $a=1$ and any $b\geq2$.  We will assume that $b-a=N>2$ with $a>1$, and that the theorem is true for $b-a<N$.

Now, using Lemma~\ref{Hecke} with $m,n$ replaced by $2^am$ and $2^{b-1}n$, we obtain
\begin{align*}a_k(2^{a-1}m,2^{b-1}n)+a_k(2^am,2)&\tau_k(2^{b-1}n)+2^{k-1}a_k(2^{a+1}m,2^{b-1}n)\cr&=a_k(2^am,2^bn)+2^{k-1}a_k(2^am,2^{b-2}n).\end{align*}
By the induction hypothesis, both terms that are multiplied by $2^{k-1}$ are divisible by $2^{k-1}2^{\rho+\gamma(N-2)}$.  Because $k-1>2\gamma$, both of these terms vanish modulo $2^{\rho+\gamma N}$.

Since $a>1$, we see that $a_k(2^am,2)$ is divisible by $2^\chi$ by Lemma~\ref{ak2in2}.  In addition, $\tau_k(2^{b-1}n)$ is divisible by $2^{\gamma(b-1)}$.  Since $\chi+\gamma(b-1)=\gamma(b-a)+\gamma(a-1)+\chi\geq\gamma(b-a)+(\gamma+\chi)\geq \rho+\gamma N$, we see that
$$a_k(2^am,2^bn)\equiv a_k(2^{a-1}m,2^{b-1}n)\pmod {2^{\rho+\gamma N}}.$$
Since the theorem is true for $a=1$ (by Lemma~\ref{ak2m2jn}), a simple induction on $a$ completes the proof.
\end{proof}
\bibliographystyle{amsplain}
\providecommand{\bysame}{\leavevmode\hbox to3em{\hrulefill}\thinspace}
\providecommand{\MR}{\relax\ifhmode\unskip\space\fi MR }
\providecommand{\MRhref}[2]{%
  \href{http://www.ams.org/mathscinet-getitem?mr=#1}{#2}
}
\providecommand{\href}[2]{#2}

\end{document}